\documentclass{amsart}

\usepackage{amsmath,amssymb,amsthm}
\usepackage{graphicx,tikz}

\newtheorem*{thm}{Theorem}

\newtheorem{lemma}{Lemma}

\theoremstyle{definition}

\theoremstyle{remark}

\begin{document}

\title[]{An upper bound on the Hot Spots Constant}
\subjclass[2020]{35B50, 35J05, 60J65.} 
\thanks{S.S. is supported by the NSF (DMS-2123224) and the Alfred P. Sloan Foundation.}

\author[]{Stefan Steinerberger}
\address{Department of Mathematics, University of Washington, Seattle, WA 98195, USA}
\email{steinerb@uw.edu}

\begin{abstract} Let $D \subset \mathbb{R}^d$ be a bounded, connected domain with smooth boundary and let $-\Delta u = \mu_1 u$
be the first nontrivial eigenfunction of the Laplace operator with Neumann boundary conditions. We prove
$$ \max_{x \in D} ~u(x) \leq 60 \cdot \max_{x \in \partial D} ~u(x)$$ 
and emphasize that this constant is uniform among all connected domains with smooth boundary in all dimensions. In particular, 
the Hot Spots Conjecture cannot fail by an arbitrary factor. The inequality also holds for other (Neumann-)eigenfunctions (possibly with a different constant) provided the eigenvalue is smaller than the first Dirichlet eigenvalue.
An example of Kleefeld shows that the optimal constant is at least $1 + 10^{-3}$.
\end{abstract}

\maketitle

\section{Introduction  }
 Let $D \subset \mathbb{R}^d$ be a bounded, connected domain with smooth boundary and let $u$ denote the first
nontrivial Laplacian eigenfunction with Neumann boundary conditions on $\partial D$, i.e.
\begin{align*}
-\Delta u &= \mu_1 u \quad \mbox{in}~ D\\
\frac{\partial u}{\partial \nu} &= 0 \quad \mbox{on}~\partial D,
\end{align*}
where $\mu_1 > 0$ and $\nu$ denotes the normal derivative on the boundary. The function $u$ the long-time behavior of generic solutions of the heat equation. 
One might expect that if $D$ is `simple', then the maximum and minimum of $u$ should be at the boundary, this is the 1974 Hot Spots Conjecture of Rauch \cite{b3}. The conjecture is widely assumed to be true for convex domains and possibly even for simply connected domains if $d=2$ -- this is being actively investigated and there are many settings where it is known to hold
 \cite{b0, atar, b3, ban, bass, beck, freitas, jerison, jerison2, judge, miya, miya2, pascu, bart, stein2}. 

\begin{center}
\begin{figure}[h!]
\begin{tikzpicture}[scale=0.8]
\draw [thick] (2,0) -- (1,1) -- (0, 0.1) -- (-1, 1) -- (-2,0);
\draw [thick] (2,0) to[out=0, in=0] (2,-2);
\draw [thick] (2,-0.1) to[out=0, in=0] (2,-1.9);
\draw [thick] (2,-2) -- (1,-3) -- (0,-2.2) -- (-1, -3) -- (-2,-2);
\draw [thick] (-2,-2) to[out=180,in=180] (-2,0);
\draw [thick] (-2,-1.9) to[out=180,in=180] (-2,-0.1);
\draw [thick] (-2, -0.1) -- (-1, -0.9) -- (0, -0.1) -- (1,-0.9) -- (2,-0.1);
\draw [thick] (2, -1.9) -- (1, -1.1) -- (0, -2.0) -- (-1,-1.1) -- (-2,-1.9);
\end{tikzpicture}
\caption{A type of domain on which Hot Spots fails (see \cite{b25, b4,kleefeld}).}
\end{figure}
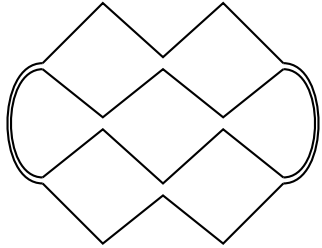
\end{center}

\section{Results} 
\subsection{Main Result.} However, there have to be at least some geometric restrictions on the domain $D$:
Burdzy \& Werner \cite{b4} constructed a planar domain with two holes where $u$ assumes its maximum strictly inside the domain. Burdzy \cite{b25} later constructed a planar counterexample with one hole. Kleefeld \cite{kleefeld} used high-precison numerics to numerically investigate examples of domains with one hole: Burdzy-type counterexamples \cite{b25} seem to be robust. Kleefeld also constructs an explicit (numerical) example of a domain for which 
$$ \|u\|_{L^{\infty}(D)} \geq (1+10^{-3}) \cdot \|u\|_{L^{\infty}(\partial D)}.$$ 
This leads to a natural question: how `wrong' can the Hot Spots conjecture be? We show that it cannot fail by more than a fixed factor.
\begin{thm}[Abridged Version]  Let $D \subset \mathbb{R}^d$ be a bounded and connected domain with smooth boundary and let $-\Delta u = \mu_1 u$
be the first nontrivial eigenfunction of the Laplace operator with Neumann boundary conditions. Then
$$ \max_{x \in D} ~u(x) \leq 60 \cdot \max_{x \in \partial D} ~u(x).$$ 
\end{thm}
\textit{Remarks.}
\begin{enumerate}
\item The optimal constant might conceivably be quite close to 1 (at least for $d=2$). An example of Kleefeld \cite{kleefeld} shows that
it is bigger than $1 + 10^{-3}$. 
\item The proof leads to better results in higher dimensions: the constant 60 can be replaced by 23 in $d=3$ dimensions and by 15 in $d=4$ dimensions. Asymptotically, as $d \rightarrow \infty$, the
constant obtained by the argument converges to $\sqrt{e^e} \sim 3.89 \dots$
\item Smoothness of the boundary can be replaced by weaker assumptions, we have not tried to optimize this condition.
\item The proof also applies for other (Neumann-)eigenfunctions $-\Delta u = \mu u$ whose eigenvalue $\mu$ is smaller than the
first Dirichlet eigenvalue (see \S 2.2).\\
\end{enumerate}

\subsection{Other eigenfunctions.}
The proof shows a slightly stronger result: we never need $-\Delta u = \mu u$ to be the \textit{first} eigenfunction of the Neumann-Laplacian, we only require that the associated eigenvalue satisfies
$$ \mu < \lambda_1 = \inf_{f \in C^{\infty}_c(D) \atop f \not\equiv 0} \frac{\int_{D} |\nabla f|^2 dx}{\int_{D} f^2 dx},$$
i.e. that the (Neumann-)eigenvalue is smaller than the first nontrivial Dirichlet eigenvalue. The inequality $\mu_1 < \lambda_1$ is elementary and follows from the variational characterization. 
Payne \cite{pay} showed for convex $D \subset \mathbb{R}^2$ with $C^2-$boundary, we also have $\mu_2 < \lambda_1$. Levine \& Weinberger \cite{levine} (see also Aviles \cite{aviles})
show that for a bounded convex domain $D \subset \mathbb{R}^d$ with sufficiently smooth boundary $\mu_d < \lambda_1$ (see also Friedlander \cite{fried3}). It is also not difficult to construct, in any dimension $d \geq 2$ and for any $n \in \mathbb{N}$, a connected domain $D_n \subset \mathbb{R}^d$ for which $\mu_n < 0.5 \cdot \lambda_1$ (think of many balls that are mutually connected by very thin tubes). 

\begin{thm}[General Version] Let $D \subset \mathbb{R}^d$ be a bounded and connected domain with smooth boundary and let $-\Delta u_k = \mu_k u_k$ be an eigenfunction of the Laplacian with Neumann boundary conditions. If $ \mu_k < \lambda_1$, then there exists a constant $0 < c < \infty$ depending only on the ratio $\mu_k/\lambda_1$ and the dimension such that
 $$ \max_{x \in D}  ~u_k(x) \leq c(\mu_k/\lambda_1,d) \cdot \max_{x \in \partial D} u_k(x).$$ 
\end{thm}
We note that the constant $c(\mu_k/\lambda_1,d)$ can be explicitly computed for any choice of parameters (we refer to the proof for details)
It could be interesting to understand whether and to what extent the limitation $\mu < \lambda_1$ is necessary or whether this could be further relaxed.

\subsection{Other equations.} There is a bigger picture here: to broad families of elliptic equations, it is possible to associate a drift-diffusion dynamical system. Eigenfunctions
of the operator diagonalize the system and have a profile under the flow which changes multiplicatively -- however, the profile also has a variational characterization from which we can infer
that the first nontrivial Neumann eigenvalue is smaller than the first nontrivial Dirichlet eigenvalue:  the associated drift-diffusion hits the boundary quickly. In order for this to
make sense in terms of the slower decay of the Neumann eigenfunction, the solution at the boundary cannot be too small. This type of reasoning can immediately applied to
more general elliptic equations and to the manifold setting -- however, getting a uniform constant will require at least some assumptions on the PDE and/or the underlying geometry. We believe this to be an 
interesting avenue for further research.

\section{Proof}
We will only give a proof of the abridged version: getting a uniform constant requires more work, the proof of the general version follows easily from the proof of the abridged version. 
 The proof decouples into the following steps.
 \begin{enumerate}
 \item \S 3.1 deduces an elementary inequality for solutions of $ -\Delta u = \mu u$ with Neumann conditions on domains $D \subset \mathbb{R}^d$ with smooth boundaries in terms of the (Dirichlet) heat kernel $p_t(\cdot, \cdot): D \times D \rightarrow \mathbb{R}$.
 \item \S 3.2 discusses an estimate on the asymptotic behavior of the optimal constant in the inequality $\mu_1 \leq c \cdot \lambda_1$ and how it depends on the dimension. This section does not require any new ideas and follows from combining existing results: the Faber-Krahn inequality \cite{faber, krahn}, the Szeg\H{o}-Weinberger inequality \cite{szego, weinberger} and an estimate of Lorch \& P. Szeg\H{o} \cite{lorch} concerning the first nontrivial root of a special function.
 \item \S 3.3 quantifies the following phenomenon: the heat kernel $p_t(\cdot, \cdot)$ decays exponentially at rate $\exp(-\lambda_1 t)$. This implies that for $\alpha < \lambda_1$, it is possible to obtain decay estimates for $e^{\alpha t} p_t(\cdot, \cdot)$. \S 3.2 will imply that we can pick $\alpha = \mu_1$ which will be useful in terms of the inequality proved in \S 3.1.
 \item \S 3.4 combines all the arguments to finish the proof.
  \end{enumerate}

\subsection{An upper bound on $\|u\|_{L^{\infty}(D)}$.} The goal of this section is to deduce a general inequality for Neumann eigenfunctions. The argument is not sensitive to the geometry of $D$, however, we require the boundary to be sufficiently regular so that reflected Brownian motion exists (smoothness more than suffices).

\begin{lemma} Let $D$ be a bounded, connected domain with smooth boundary and let $u$ be a nontrivial solution of $-\Delta u = \mu u$ with Neumann boundary conditions. Then, for all $t>0$,
$$1 \leq e^{\mu t}\int_{D} p_{t}(x_0, y)  dy + e^{\mu t}\left(1 -  \int_{D} p_{t}(x_0, y) dy\right)  \frac{\|u\|_{L^{\infty}(\partial D)}}{\|u\|_{L^{\infty}(D)}},$$
where $p_t(x,y)$ denotes the Dirichlet heat kernel on $D$.
\end{lemma}
\begin{proof}
Let us
assume without loss of generality (after possibly replacing $u$ by $-u$) that $u$ assumes its maximum inside the
domain $D$ and
$$ u(x_0) = \| u\|_{L^{\infty}(D)} > 0.$$
We will show that the maximum value on the boundary cannot be too small.
Solving the heat equation with $u$ as initial datum and Neumann boundary conditions is simple because $u$ is an eigenfunction, the solution is
$ u(t,x) = e^{-\mu t} u(x).$
The probabilistic interpretation of the heat equation then allows us to write
$$ u(t,x) = \mathbb{E}_x u(\omega(t)),$$
where the expectation $\mathbb{E}_x$ is taken with respect to Brownian motions started in $x$ and reflected on the boundary.
We now fix a time $t > 0$ and use the probabilistic interpretation to derive an upper bound on $u(x_0) = \|u\|_{L^{\infty}(D)}$.
Let us consider a fixed Brownian path $\omega_{x_0}(s)$ started in $x_0$ and running for time $0 \leq s \leq t$. We distinguish
two cases
\begin{enumerate}
\item $\omega_{x_0}(s)$ never touches the boundary for all $0 \leq s \leq t$
\item there exists a first time $0 \leq t_0 \leq  t$ such that $\omega_{x_0}(t_0) \in \partial D$.
\end{enumerate}
The first case is naturally related to the evolution of the heat equation with \textit{Dirichlet} boundary conditions
\begin{align*}
\frac{\partial u}{\partial t} &= \Delta u \qquad \mbox{in}~D \\
u &=0 \qquad \quad \mbox{on}~\partial D.
\end{align*}
We recall that the Dirichlet heat equation with initial data $u(0,x) = f(x)$ has an explicit solution
$$ u(t,x) = \int_{D} p_t(x,y) f(y) dy,$$
where
$p_t:D \times D \rightarrow \mathbb{R}$ is the heat kernel associated to the Dirichlet heat equation on $D$. Using $(\phi_n)_{n=1}^{\infty}$ to denote the sequence of $L^2-$normalized eigenfunctions of
$ -\Delta \phi_n = \lambda_n \phi_n$
with Dirichlet boundary conditions, we have an explicit expression for the heat kernel
$$ p_t(x,y) = \sum_{n=1}^{\infty} e^{-\lambda_n t} \phi_n(x) \phi_n(y).$$
At the same time, the Dirichlet heat equation has an explicit stochastic representation as
$$ u(t,x) =\mathbb{E}_{w} u(\omega(t)),$$
where $\omega(t)$ is a Brownian motion started in $x$ and running for time $t$ with the additional rule that, once hitting the boundary, it remains there for all time (the boundary is `sticky').
Since we are dealing with zero boundary conditions, the particles at the boundary do not contribute to the expectation and the expectation is only taken with respect to Brownian
motion that never touches the boundary. This reproduces exactly the first case in our case distinction. Moreover, it implies that the density of the particles who have never touched the boundary is given by $p_t(x_0, \cdot)$
and the likelihood of Brownian motion starting in $x_0$, running for time $t$ and never hitting the boundary is given by
$$ \mathbb{P}\left(\forall~0 \leq s \leq t: \omega_{x_0}(s) \notin \partial D\right) = \int_{D} p_{t}(x_0, y) dy.$$
This shows that we have a very precise understanding of the first case. As for the second case, we will not
derive such a precise representation and instead bound it from above: suppose $\omega_{x_0}(t_0) \in \partial \mathbb{D}$.
We can use Markovianity and the property that $u$ is an eigenfunction to take an expectation over all Brownian
paths started in $t_0$ and running for $t-t_0$ units of time:
$$ \mathbb{E}_{\omega}~ \omega(t-t_0) = e^{-\mu (t-t_0)} u(\omega_{x_0}(t_0)) \leq \|u\|_{L^{\infty}(\partial D)}.$$
Combining both estimates, we have
\begin{align*}
 u(t, x_0) &\leq  \int_{D} p_{t}(x_0, y) u(y) dy + \left(1 -  \int_{D} p_{t}(x_0, y) dy\right)  \|u\|_{L^{\infty}(\partial D)} \\
 &\leq \|u\|_{L^{\infty}(D)} \int_{D} p_{t}(x_0, y)  dy + \left(1 -  \int_{D} p_{t}(x_0, y) dy\right)  \|u\|_{L^{\infty}(\partial D)}.
 \end{align*}
 Plugging in the explicit solution of the heat equation $u(t,x) = e^{-\mu t} u(x)$ and using that $u(x_0) = \|u\|_{L^{\infty}(D)}$, we obtain the desired inequality
 \begin{align*}
\|u\|_{L^{\infty}(D)} &\leq e^{\mu t} \|u\|_{L^{\infty}(D)} \int_{D} p_{t}(x_0, y)  dy + e^{\mu t}\left(1 -  \int_{D} p_{t}(x_0, y) dy\right)  \|u\|_{L^{\infty}(\partial D)}.
 \end{align*}
 which can also be written as
  \begin{align*}
1 &\leq  e^{\mu t}\int_{D} p_{t}(x_0, y)  dy + e^{\mu t}\left(1 -  \int_{D} p_{t}(x_0, y) dy\right)  \frac{\|u\|_{L^{\infty}(\partial D)}}{\|u\|_{L^{\infty}(D)}}.
 \end{align*}
\end{proof}
It is clear from the form of the inequality that if we can find a value of $t$ such that
$$  e^{\mu t}\int_{D} p_{t}(x_0, y)  dy < 1,$$
then this inequality implies that the maximum at the boundary cannot be too small compared to the maximum inside the domain. Moreover, for each
domain $D$, 
$$ p_t(x,y) = \sum_{k=1}^{\infty} e^{-\lambda_k t} \phi_k(x) \phi_k(y)$$
decays asymptotically like $\sim e^{-\lambda_1 t}$ and since $\mu < \lambda_1$, there always exist a $t > 0$ such that the desired quantity is
less than 1. It remains to show that all these quantities can be chosen uniformly among all domains $D$.

 \subsection{Inequalities for Eigenvalues.} We discuss an inequality for eigenvalues due to Szeg\H{o} \cite{szego} (refining earlier work of Polya \cite{pol}).
 These types of inequalities are now well-studied, we refer to \cite{ash, ash3, fil, fried, korn, levine, payne, stein}.
 
 \begin{lemma}[Szeg\H{o}, 1954] If $D$ is a planar domain, then
  $$ \mu_1 \leq 0.587 \cdot \lambda_1.$$
 \end{lemma}
 \begin{proof}
 The Faber-Krahn inequality \cite{faber, krahn} says that among all domains with fixed volume the ball minimizes the first eigenvalue of the Dirichlet-Laplacian and thus
 $$ \lambda_1 \geq \frac{\pi \cdot j_{0,1}^2}{|D|},$$
 where $j_{0,1} \sim 2.404\dots$ is the first positive zero of the Bessel function $J_0$. The case of equality is given when $D$ is a disk.
 There is a corresponding upper bound for $\mu_1$, the first nontrivial eigenvalue of Laplacian with Neumann boundary condition 
 $$ \mu_1 \leq \frac{\pi \cdot \widehat j_{1,1}^2}{|D|},$$
 where $\widehat j_{1,1} \sim 1.841$ is the first positive zero of $J_{1}$. This inequality was first proven by Szeg\H{o} \cite{szego} for simply connected
 planar domains and then more generally by Weinberger \cite{weinberger} who also established the corresponding analogue in higher dimensions. Equality
 is again attained on the ball.
\end{proof} 

A corresponding inequality also exists in higher dimensions with an improved constant. The argument is essentially the same: the new ingredient is the need to control roots of certain Bessel functions.
\begin{lemma}
Let $D \subset \mathbb{R}^d$ and $d \geq 3$. We have
  $$ \mu_1 \leq  \min\left\{0.587, \frac{d+2}{\left(\frac{d}{2}-\frac{2}{d}\right)^2}  \right\} \cdot \lambda_1.$$
\end{lemma}
\begin{proof}
Using again the Faber-Krahn inequality \cite{faber, krahn} and the Szeg\H{o}-Weinberger
inequality \cite{szego, weinberger}, we have
  $$ \mu_1 \leq \frac{p_{d/2,1}^2}{j_{d/2-1,1}^2} \cdot \lambda_1,$$
  where $j_{d/2-1,1}$ is the first positive root of the Bessel function $J_{{d/2-1}}(x)$ and
  $p_{d/2,1}$ is the first positive root of the derivative of $x^{1-d/2} J_{d/2}(x)$. We refer to
  a paper of Ashbaugh \& Benguria \cite{ash3} which gives a clear exposition of these inequalities.
We have the known estimate (see e.g. Watson \cite{watson})
$$ j_{d/2-1,1} \geq \sqrt{(d/2-1)(d/2+1)} \geq \frac{d}{2} - \frac{2}{d}.$$
It remains to obtain an upper bound on $p_{d/2, 1}^2$. We could invoke bounds of Lorch \& Szeg\H{o} \cite{lorch} which guarantee
$$ d + \frac{8}{d+6} < p_{d/2,1}^2 < d + 2$$
and this proves the desired inequality. However, we can also give an elementary and self-contained proof of the upper bound.
The Szeg\H{o}-Weinberger inequality
$$ \mu_1(D) \leq \left( \frac{c_d}{|D|}\right)^{\frac{2}{d}} \cdot p_{d/2,1}^2,$$
where $c_d$ is the volume of the unit ball, is attained on the unit ball $\mathbb{B}_1$.
Invoking the variational characterization, we have for any $f:\mathbb{B}_1 \rightarrow \mathbb{R}$ with
mean value 0 
$$\mu_1(\mathbb{B}_1) =  p_{d/2,1}^2 \leq \frac{\int_{\mathbb{B}_1} | \nabla f(x)|^2 dx}{\int_{\mathbb{B}_1}  f(x)^2 dx}.$$
This allows us to obtain upper bounds on $\mu_1$ by plugging in test functions.
Assuming $\mathbb{B}_1$ to be centered at the origin, we choose $f(x) = x_1$. Applying Fubini, we get, for a
dimensional constant $\omega_d>0$,
$$ \int_{\mathbb{B}_1} | \nabla f(x)|^2 dx = \int_{\mathbb{B}_1} 1~  dx =\omega_d \int_{-1}^{1} \left(1-x_1^2\right)^{\frac{d-1}{2}} ~ d x_1$$
as well as
$$ \int_{\mathbb{B}_1}  f(x)^2 dx = \omega_d  \int_{-1}^{1} \left(1-x_1^2\right)^{\frac{d-1}{2}} x_1^2 ~dx_1.$$
The constant $\omega_d$ cancels and we arrive at
$p_{d/2,1}^2 \leq d+2.$
\end{proof}

\subsection{A Decay Estimate.} The goal of this section is to derive an upper bound on 
$$\mbox{the quantity} \qquad e^{\mu_1 t}\int_{D} p_{t}(x_0, y)  dy,$$
where $x_0 \in D$ is the point in which $u$ assumes its maximum.
We will not use the fact that $x_0$ is the point in which $u$ assumes its maximum and will instead bound the bigger quantity
$$ e^{\mu_1 t}\int_{D} p_{t}(x_0, y)  dy \leq  \sup_{x \in D} e^{\mu_1 t} \int_{D} p_{t}(x, y)  dy.$$
Moreover, the argument does not require that $\mu_1$ is the first Neumann eigenvalue, it only requires that $\mu_1/\lambda_1 < 1$.
\begin{lemma} Let $D$ be as above and $\mu < \lambda_1$. Then 
$$   \sup_{x \in D} \int_{D} e^{\mu t} p_t(x,y) dy \leq \frac{|D|}{(4 (1-\frac{\mu}{\lambda_1}) \pi t)^{d/2}}.$$
\end{lemma}
\begin{proof} We only consider the special case $\mu = \mu_1$ (which then allows us to invoke uniform bounds on $\mu_1/\lambda_1$ which depend only on the dimension) but it is easy to see that nothing changes in the proof in the general case.
We recall that
$$ p_t(x,y) = \sum_{k=1}^{\infty} e^{-\lambda_k t} \phi_k(x) \phi_k(y)$$
and that therefore
$$ \int_{D} e^{\mu_1 t} p_t(x,y) dy =  \sum_{k=1}^{\infty} e^{(\mu_1 -\lambda_k) t} \phi_k(x)  \int_{D} \phi_k(y) dy.$$
Using
$$ \left( \int_{D} \phi_k(y) \right)^2 \leq |D| \cdot \int_{D} \phi_k(y)^2 dy = |D|.$$
in combination with Cauchy-Schwarz leads to
\begin{align*}
 \sum_{k=1}^{\infty} e^{(\mu_1 -\lambda_k) t} \phi_k(x)  \int_{D} \phi_k(y) dy   &\leq \left(\sum_{k=1}^{\infty} e^{(\mu_1 -\lambda_k) t} \phi_k(x)^2 \right)^{1/2}  \\
 &\cdot \left( \sum_{k=1}^{\infty}e^{(\mu_1 -\lambda_k) t} \cdot |D|  \right)^{1/2}.
 \end{align*}
 This requires us to bound 
 $$ \sup_{x \in D} \sum_{k=1}^{\infty} e^{(\mu_1 -\lambda_k) t} \phi_k(x)^2 \qquad \mbox{and} \qquad \sup_{x \in D}  |D| \cdot \sum_{k=1}^{\infty}e^{(\mu_1 -\lambda_k) t}.$$
 For both these terms, the dimension of $D$ will start playing a role. Recalling Lemma 3, we will denote the best constant in the inequality by
  $$ \mu_1 \leq \alpha_d \cdot \lambda_1.$$
We note that $\alpha_d \leq 0.587$ and that $\alpha_d$ converges to 0 as $d \rightarrow \infty$. Then
 \begin{align*}
  \frac{\mu_1 - \lambda_k}{\lambda_k} &= \frac{\mu_1}{\lambda_k} - 1 \leq  \frac{\mu_1}{\lambda_1} - 1 \leq -(1-\alpha_{d}).
  \end{align*}
 We can now bound the first term by
\begin{align*}
  \sup_{x \in D} ~\sum_{k=1}^{\infty} e^{(\mu_1 -\lambda_k) t} \phi_k(x)^2  &\leq \sup_{x \in D}~ \sum_{k=1}^{\infty} e^{-\lambda_k (1-\alpha_d) t} \phi_k(x)^2 \\
  &= \sup_{x \in D}~ p_{(1-\alpha_d)t}(t,x,x).
 \end{align*}
 At this point we invoke domain monotonicity of the heat kernel: $p_t(\cdot, \cdot)$, the heat kernel associated to the domain $D$, is bounded from above by
 the heat kernel in $\mathbb{R}^d$ 
 $$ p_t(x,y) \leq \frac{1}{(4 \pi t)^{d/2}} \exp\left(-\frac{\|x-y\|^2}{4t}\right) \leq  \frac{1}{(4 \pi t)^{d/2}}$$
 and therefore
 $$   \sup_{x \in D} \sum_{k=1}^{\infty} e^{(\mu_1 -\lambda_k) t} \phi_k(x)^2  \leq  \frac{1}{(4 (1-\alpha_d) \pi t)^{d/2}}.$$
 A similar argument can be used for the second term: using the $L^2-$normalization of the eigenfunctions shows
\begin{align*}
|D| \cdot  \sum_{k=1}^{\infty}e^{(\mu_1 -\lambda_k) t} &= |D| \cdot  \int_{D}  \sum_{k=1}^{\infty}e^{(\mu_1 -\lambda_k) t} \phi_k(x)^2 dx \\
 &\leq |D| \cdot  \int_{D}  p_{(1-\alpha_d)t}(t,x,x) dx \leq |D|^2 \cdot \sup_{x \in D}  p_{ (1-\alpha_d)t}(t,x,x),
 \end{align*}
 which can be bounded by the same term as above.
Combining all these estimates, we arrive at
\begin{align*}
  \sup_{x \in D} \int_{D} e^{\mu_1 t} p_t(x,y) dy  \leq \frac{|D|}{(4 (1-\alpha_d) \pi t)^{d/2}}.
 \end{align*}
 \end{proof}

 \subsection{Finishing the Argument.}
 Recalling Lemma 1,
 \begin{align*}
 1 &\leq  e^{\mu_1 t}\int_{D} p_{t}(x_0, y)  dy +\left( e^{\mu_1 t} -   e^{\mu_1 t}\int_{D} p_{t}(x_0, y) dy\right)  \frac{\|u\|_{L^{\infty}(\partial D)}}{\|u\|_{L^{\infty}(D)}}
 \end{align*}
we observe that in the case of interest to us, the quantity $ \|u\|_{L^{\infty}(\partial D)}/\|u\|_{L^{\infty}(D)}$ is strictly smaller than 1 and we also note that
$$ e^{\mu_1 t}\int_{D} p_{t}(x_0, y)  dy +\left( e^{\mu_1 t} -   e^{\mu_1 t}\int_{D} p_{t}(x_0, y) dy\right) = e^{\mu_1 t}.$$
 Thus, replacing the first term by an upper bound on the first term (both times that it appears in the inequality), leads to an upper bound. Applying this together with Lemma 4,
 we end up with
 \begin{align*}
 1 &\leq   \frac{|D|}{(4 (1-\alpha_d) \pi t)^{d/2}}+ \left(e^{\mu_1 t} -  \frac{|D|}{(4 (1-\alpha_d) \pi t)^{d/2}} \right)  \frac{\|u\|_{L^{\infty}(\partial D)}}{\|u\|_{L^{\infty}(D)}}.
\end{align*}
We use the Szeg\H{o}-Weinberger inequality once more in the form
$$ \mu_1(D) \leq  \frac{c_d^{2/d}}{|D|^{2/d}} \cdot p_{d/2,1}^2.$$
At this point, we can invoke a scaling symmetry which has not been used so far -- we can exchange a dilation of time with a rescaling of time. This is reflected in
the fact that powers of $t$ and $|D|$ are not independent: we exploit this by introducing a rescaling of time, $\alpha = t |D|^{-2/d}$. The inequality then simplifies to,
for all $\alpha > 0$,
 \begin{align*}
 1 &\leq   \frac{1}{(4 (1-\alpha_d) \pi \alpha)^{d/2}}+ \left(\exp\left( c_d^{\frac{2}{d}} \cdot p_{d/2,1}^2 \alpha \right) -  \frac{1}{(4 (1-\alpha_d) \pi \alpha)^{d/2}} \right)  \frac{\|u\|_{L^{\infty}(\partial D)}}{\|u\|_{L^{\infty}(D)}}.
\end{align*}
This inequality can now be optimized in $\alpha$, the results will depend on the values of $\alpha_d$ and $p_{d/2,1}^2$. We also recall that
$ c_d = \pi^{d/2}/\Gamma\left(d/2 + 1\right)$
is the volume of the $d-$dimensional unit ball.
 We discuss the first few cases which will illustrate the general pattern. In the planar case, $d=2$, we have $p_{d/2,1} \sim 1.841$ and $\alpha_d \sim 0.587$.
Setting $\alpha = 0.258$, we obtain
$$  \frac{\|u\|_{L^{\infty}(\partial D)}}{\|u\|_{L^{\infty}(D)}} \geq \frac{1}{58.35}.$$\\
For $d=3$, we have $ p_{d/2,1} \sim 2.081$ and $ \alpha_d \sim 0.439$.
Setting $\alpha = 0.194$, we obtain
$$  \frac{\|u\|_{L^{\infty}(\partial D)}}{\|u\|_{L^{\infty}(D)}} \geq \frac{1}{22.03}. $$
For $d=4$, we have $ p_{d/2,1} \sim 2.299$ and  $\alpha_d \sim 0.3602$.
Setting $\alpha = 0.17$, we obtain
$$  \frac{\|u\|_{L^{\infty}(\partial D)}}{\|u\|_{L^{\infty}(D)}} \geq \frac{1}{14.71}. $$
The pattern continues in the same manner for $d=5,6,\dots$ and the constants
keep improving. It remains to understand the asymptotic behavior.
As $d \rightarrow \infty,$ we have $\alpha_d \rightarrow 0$ and hence for any fixed
$ \alpha > 1/(4\pi),$
one of the terms become small
$$ \lim_{d \rightarrow \infty}  \frac{1}{(4 (1-\alpha_d) \pi \alpha)^{d/2}} = 0.$$
Ignoring these terms (which can be made precise since they become sufficiently small for $d$ sufficiently large), using the precise form of $c_d$ and Lemma 3
\begin{align*}
 \frac{\|u\|_{L^{\infty}(\partial D)}}{\|u\|_{L^{\infty}(D)}} &\geq \exp\left( -c_d^{\frac{2}{d}} \cdot p_{d/2,1}^2 \alpha \right) \geq \exp\left(  - \frac{\pi}{\Gamma\left(\frac{d}{2} + 1\right)^{2/d}} \cdot (d+2) \alpha \right).
 \end{align*}
 The arising sequence in $d$ has a beneficial monotonicity property as well as a nicely defined limit which simplifies our reasoning by invoking
$$ \frac{\pi}{\Gamma\left(\frac{d}{2} + 1\right)^{2/d}} \cdot (d+2) \leq \lim_{d \rightarrow \infty} \frac{\pi}{\Gamma\left(\frac{d}{2} + 1\right)^{2/d}} \cdot (d+2)  = 2 e \pi$$
and therefore, invoking that $\alpha> 1/4\pi$ was arbitary
$$ \exp\left(  - \frac{\pi}{\Gamma\left(\frac{d}{2} + 1\right)^{2/d}} \cdot (d+2) \alpha \right) \geq \exp\left(  - 2e\pi \frac{1}{4\pi} \right)  = \sqrt{e^{-e}} = \frac{1}{3.89\dots}.$$
This is the best constant that can be achieved with this type of argument. 

\subsection{A Final Remark.}
There is one step in the proof that is guaranteed to be wasteful: when evaluating the second sum in \S 3.3, we have the inequalities
\begin{align*}
|D| \cdot  \sum_{k=1}^{\infty}e^{(\mu_1 -\lambda_k) t} &= |D| \cdot  \int_{D}  \sum_{k=1}^{\infty}e^{(\mu_1 -\lambda_k) t} \phi_k(x)^2 dx \\
 &\leq |D| \cdot  \int_{D}  p_{(1-\alpha_d)t}(x,x) dx \leq |D|^2 \cdot \sup_{x \in D}  p_{ (1-\alpha_d)t}(x,x).
 \end{align*}
 It is clear that this estimate is wasteful since an integral is replaced by a pointwise supremum. Indeed, the integral
 $$  \int_{D}  p_{t}(x,x) dx \qquad \mbox{is known as the heat trace.}$$
Luttinger \cite{lutt2, lutt3} proved that among all domains with fixed volume, the heat trace is maximized by the ball which would lead to
a better estimate. However, we are not aware of any
reasonably explicit estimates for the heat trace of the ball. Asymptotic expansions are available (see e.g. Davies \cite{davies}),
however, our argument would require explicit bounds.

\end{document}